\documentclass{amsart}

\usepackage{amsmath}
\usepackage{amsthm}
\usepackage{amssymb}
\usepackage[all]{xy}

\newtheorem{theorem}{Theorem}[section]

\newtheorem{proposition}[theorem]{Proposition}
\newtheorem{corollary}[theorem]{Corollary}

\theoremstyle{definition}
\newtheorem{definition}[theorem]{Definition}

\newcommand{\B}{\ensuremath{\mathbf{B}}}
\newcommand{\A}{\ensuremath{\mathbf{A}}}
\newcommand{\Cat}{\ensuremath{\mathsf{Cat}}}
\newcommand{\Set}{\ensuremath{\mathsf{Set}}}
\newcommand{\Grp}{\ensuremath{\mathsf{Grp}}}
\newcommand{\G}{\ensuremath{\mathsf{G}}}
\newcommand{\Grpd}{\ensuremath{\mathsf{Grpd}}}
\newcommand{\C}{\ensuremath{\mathbf{C}}}
\renewcommand{\a}{\alpha}
\renewcommand{\b}{\beta}
\newcommand{\g}{\gamma}
\renewcommand{\l}{\lambda}
\renewcommand{\r}{\rho}
\DeclareMathOperator{\conj}{conj}
\DeclareMathOperator{\dom}{dom}
\DeclareMathOperator{\cod}{cod}

\begin{document}

\title{Internal bicategories in groups}
\author[N. Martins-Ferreira]{Nelson Martins-Ferreira}
\address{School of Technology and Management-ESTG\\
Centre for Rapid and Sustainable Product Development-CDRSP\\
Polytechnic Institute of Leiria\\
P-2411-901, Leiria, Portugal}
\email{martins.ferreira@ipleiria.pt}



\thanks{Thanks are due to A. Montoli and M. Sobral for many helpful comments and suggestions on the text. This research was supported by the
FCT post doctoral grant SFRH/BPD/43216/2008 at CMUC, by the FCT projects PTDC/\-MAT/120222/2010
and PTDC/EME-CRO/120585/2010, and also by CDRSP and ESTG from the Polytechnic Institute of Leiria.}

\begin{abstract}
A detailed description of internal bicategory in the category of groups is derived from the general description of internal bicategories in weakly Mal'tsev sesquicategories. The example of bicategory of paths in a topological abelian group is presented, as well as
a description of internal bicategories in weakly Mal'tsev categories with a $V$-Mal'tsev operation in the sense of Pedicchio.
\end{abstract}

\maketitle

\section{Introduction}

The purpose of this note is to give a detailed description of internal
bicategories (with the restriction $1_a\otimes 1_a=1_a$, instead of the usual condition $1_a\otimes 1_a\cong1_a$) in the category of groups.
 An example that illustrates the structure in the simpler case of abelian groups is given in section 4. An interesting characteristic of internal bicategories in groups is the fact that the associativity isomorphism ${a\otimes(b\otimes c)\cong (a\otimes b)\otimes c}$ is completely determined by the left and right isomorphisms for the composition with pseudo-identities.

For the purpose of this note, an internal bicategory \cite{Benabou:Bicategories} is best described
 as a particular instance of a
pseudocategory \cite{NMF-PsCat} as follows: if $\B$ is an  arbitrary category
with pullbacks and $\Cat(\B)$ is the 2-category of internal
categories in $\B$ then an internal bicategory in $\B$ is a pseudocategory in $\Cat(\B)$ such that the \emph{object of objects} is a discrete internal  category.

Pseudocategories internal to a weakly Mal'tsev category with a weakly Mal'tsev cartesian 2-cell structure were characterized in \cite{NMF-PhD} (see also Theorem 6.2 in \cite{NMF-2cells}). If $\B$ is a weakly Mal'tsev category then so it is the category of internal categories in $\B$, $\Cat(\B)$, which, in addition, has a canonical 2-cell structure given by the internal natural transformations that makes $\Cat(\B)$ a weakly Mal'tsev category with a weakly Mal'tsev cartesian 2-cell structure. 
It is well known that the category $\Grp$, of groups and group homomorphisms, is a (weakly) Mal'tsev category and consequently $\Cat(\Grp)$, which is equivalent to the category of crossed-modules, with homotopies as 2-cells (see e.g. \cite{NMF-2cells}), is an example of a weakly Mal'tsev category with a weakly Mal'tsev cartesian 2-cell structure.

A pseudocategory  in a
2-category $\C$  can be seen as an internal structure in $\C$ of the form 
\[\xymatrix{C_2
\ar@<2.5ex>[r]^{\pi_2}\ar@<-2.5ex>[r]_{\pi_1}\ar@<0ex>[r]|{m} & C_1
\ar@<1ex>[r]^{d}\ar@<-1ex>[r]_{c}\ar@<-1.5ex>[l]|{e_2}\ar@<1.5ex>[l]|{e_1}
& C_0 \ar@<0ex>[l]|{e} }\] in which $C_2$ is defined together with
$\pi_1$ and $\pi_2$ as the pullback  of $d$ and $c$, the morphisms $e_1$ and $e_2$ are uniquely determined as $e_1=\langle
1,ed\rangle$ and $e_2=\langle ec,1\rangle$, the conditions  $de=1_{C_0}=ce$,
$dm=d\pi_2$ and $cm=c\pi_1$ hold, while the usual axioms that regulate
associativity  and identity of composition, namely $m(1_{C_1}\times_{C_0} m)=m(m\times_{C_0} 1_{C1})$ and $me_2=1_{C_1}=me_1$, are replaced by the
existence of appropriate (iso) 2-cells 
\[\xymatrix{ m(1\times m) \ar@{=>}[rr]^{\alpha} && m(m\times 1) \\
me_2 & 1_{C_1} \ar@{=>}[l]_{\lambda} \ar@{=>}[r]^{\rho} & me_1}\]
with $d\lambda=d=d\rho$, $c\lambda=c=c\rho$,  $d\alpha=d\pi_3$,
$c\alpha=c\pi_1$, $\lambda e=\rho e$, and furthermore they are required to satisfy coherence conditions (see e.g. \cite{NMF-PsCat,NMF-2cells} for further details).

In particular if $\C=\Cat(\B)$ for some category $\B$, and if $C_0$ in the diagram above is a
discrete internal category, then the resulting notion is precisely that of
an internal bicategory in  $\B$.

We give an explicit description of internal bicategories in the category of groups, $\Grp$, subject to two conditions: the  object $C_1$ is an internal
groupoid (which corresponds to the requirement that the \emph{2-cells} in
the bicategoy are invertible) and $\lambda
e=e=\rho e$ is satisfied (which means that
$1_a\otimes 1_a=1_a$ for every identity arrow $1_a$). The first condition is not a
restriction when $\B$ is the category of groups since every internal category there is the same
as an internal groupoid. The
second condition is considered for instance in
\cite{Grandis-Pare} and it holds for the bicategory of paths in an arbitrary topological space.

In the last two sections we explore the more general context of a weakly Mal'tsev category with a $V$-Mal'tsev operation in the sense of Pedicchio \cite{Pedicchio}.  In \cite{Pedicchio} it is shown that if $V\colon{\A\to\B}$ is conservative and there exists a $V$-Mal'tsev operation $p\colon{V^3\to V}$ then $\A$ is a  Mal'tsev category. If we take the functor $V$ to be faithful but not necessarily conservative then $\A$ is a weakly Mal'tsev category.  In this context it is no longer true, in general, that internal categories and internal groupoids coincide, hence we have to restrict our study to the category $\Grpd(\A)$ of internal groupoids in $\A$. However, instead of $\Grpd(\A)$ we work with an equivalent category, denoted by $\G(\A)^{*}$, whose objects are the triples $(A,s,t)$ with $A$ an object in $\A$ and $s,t\colon{A\to A}$ two endomorphisms of $A$ with $st=t$, $ts=s$ and such that the property (P1) is satisfied (see Section \ref{sec:5}).

 Finally, we consider the category $\G(\A)^{*}$, equipped with the canonical 2-cell structure which corresponds to the internal natural transformations in $\Grpd(\A)$ and characterize the internal pseudocategories in $\Grpd(\A)$. As a consequence we obtain a characterization for the internal bicategories in $\A$ which have invertible 2-cells and satisfy the homogeneity condition $1_a\otimes 1_a=1_a$ as explained before. We remark that the description obtained cannot be improved, in general, though, in the case of groups, a relevant simplification arises because the property (P1) is equivalent to the following (simpler) condition:
\begin{enumerate}
\item[(P1')] there exists a morphism $\sigma\colon{A\to A}$ for which $V(\sigma)=p_A\langle V(s), 1, V(t)\rangle$.
\end{enumerate}
This equivalence comes from a coincidence of commutators (see e.g. \cite{MF+VdL:14online}).

\section{Internal bicategories in $\Grp$}

In the category of groups,  Theorem~6.2 in \cite{NMF-2cells}, which is proved in \cite{NMF-PhD}, gives
the result we state next.

Every internal bicategory (with $1_a\otimes 1_a=1_a$, as it is the
case for the bicategory of paths in a space) in the category of
groups is completely determined, up to isomorphism, by the
following data
\begin{equation}\label{d.st-lr-dec}
\xymatrix{ X   \ar@<0ex>[r]^{h}  \ar@(ld,d)[]_{ \lambda,\rho }
\ar@(lu,u)[]^{ s,t }
 & B  },\quad  \xi\colon{B\times X\to X}
\end{equation}
with $X$ and $B$ groups, $h,s,t,\lambda,\rho$ group  homomorphisms
and $\xi$ a set theoretical map satisfying the following
conditions:
\begin{enumerate}
\item [(i)] $\xi$ represents an (external) action of the multiplicative group $B$ on the additive (but not necessarily abelian) group $X$ , written as $\xi(b,x)=b\cdot x$;
\item [(ii)] $h(b\cdot x)=bh(x)b^{-1}$;
\item[(iii)] $st=t$, $ts=s$;
\item[(iv)] $hs=h=ht$;
\item[(v)] $h\lambda=h=h\rho$;
\item[(vi)] $\l-\l s+t\l s=1_X=\r-\r s+t\r s$;
\item[(vii)] $[\ker (s),\ker (t)]=0$;
\item[(viii)] the two maps \[m\colon{X^{2}\times B\to X\times B}\longleftarrow X^{3}\times B\colon{\alpha}\] defined by
\begin{eqnarray}
m(x,x',b)=(u(x)+v(x'),b) \nonumber\\
\alpha(x,x',x'',b)=(\a_1(x)+\a_2(x')+\a_3(x''),b)
\end{eqnarray}
with
\begin{eqnarray}
u=s\r t-\r t+ \r \label{eq: u}\\
v=s\l t-\l t+\l \label{eq: v}\\
\a_1=s\r\r t-\r s\r t+ s\r t-\r t+\r  \label{eq: a1}\\
\a_2=s\r \l - \r s \l t + s \r t - \r t + \r \label{eq: a2}\\
\a_3= \l s \l t - \l \l t+ \l\l \label{eq: a3}
\end{eqnarray}
are homomorphisms between iterated semidirect products
\[X\rtimes_{\xi_1}(X\rtimes_{\xi}B)\to X\rtimes_{\xi}B
\longleftarrow X\rtimes_{\xi_2}(X\rtimes_{\xi_1}(X\rtimes_{\xi}B))\]
with $\xi_1((x',b),x)=(h(x')b)\cdot x$ and
$\xi_2((x',x'',b),x)=(h(x'x'')b)\cdot x$.
\end{enumerate}

Some useful remarks and observations:
\begin{enumerate}
\item  On item (i), by an (external) action, we mean that $\xi$ may also be seen as a family of group homomorphisms $(\xi_b\colon{X\to X})_{b\in B}$, indexed by the elements $b\in B$, such that $\xi_b\xi_{b'}=\xi_{bb'}$ and $\xi_{1}=1_X$. More explicitly, considering $B$ with the multiplicative notation, $X$ with the additive notation and writing $\xi_b(x)=b\cdot x$, we have:
\begin{enumerate}
\item $1\cdot x=x$
\item $(bb')\cdot x=b\cdot(b\cdot x)$
\item $b\cdot(x+y)=b\cdot x + b\cdot y$
\end{enumerate}
\item On item (ii), if we let $\conj_b\colon{B\to B}$ to be the conjugation action, which associates  $byb^{-1}$ to every $y\in B$, then this condition precisely says that \[h\xi_b=\conj_bh, \quad \text{for all }b\in B.\]
\item On item (vi) there is an obvious abuse of notation: it clearly means \[\l(x)-\l s(x)+t\l s(x)=x=\r(x)-\r s(x)+t\r s(x)\]for every $x\in X$. The same is true for conditions $(\ref{eq: u})$--$(\ref{eq: a3})$.
\item Finally, the condition of item (vii) is saying that the elements of $\ker(s)$ commute with the ones of $\ker(t)$ (see \cite{MacLane}, p.285 and followings).
\end{enumerate}

Remark that if $s=t=\r=\l=1_X$ then the resulting structure is precisely the one of a crossed module (see e.g. \cite{MacLane}, p.285). Indeed, in this case, $m$ is 
a homomorphism if
and
only if $x+x'-x=h(x)\cdot x'$, for all $x,x'\in X$, while the other
conditions trivialize.

The fact that every internal bicategory in the category of groups has, up to an isomorphism of structures, the form presented above is a consequence of Theorem 6.2 as stated in \cite{NMF-2cells} whose proof can be found in \cite{NMF-PhD}. In the following section we give a more explicit description of internal bicategories in groups.

\section{An extended description}

In order to compare the  previous description of an internal bicategory in groups with
the well-known description of internal  2-category in groups, we rewrite conditions (iii), (iv) and (vii) using the fact that, up to isomorphism,
$X\simeq Y\rtimes_\phi Z$. From this point of view, $$s,t\colon{Y\rtimes_\phi Z\to
Y\rtimes_\phi Z}$$  are given  by $s(y,z)=(0,z)$ and
$t(y,z)=(0,\partial_{1}(y)z)$ with  $\partial_{1}\colon{Y\to Z}$ a homomorphism and 
$\phi\colon{Z\times Y\to Y}$ a map being part of a crossed  module
structure.

Conditions (iv) to (vi) now read as
\[h\simeq\left(
    \begin{array}{cc}
      0 & \partial_0
    \end{array}
  \right),\quad \l\simeq\left(
    \begin{array}{cc}
      1 & \l_0 \\
      0 & 1-\partial_{1}\l_0 \\
    \end{array}\right),\quad \r\simeq\left(
    \begin{array}{cc}
      1 & \r_0 \\
      0 & 1-\partial_{1}\r_0 \\
    \end{array}
  \right),
\]
and the  basic structure is transformed into
\begin{equation}\label{diag-1}
\xymatrix{Y \ar[r]^{\partial_{1}} & Z  \ar[r]^{\partial_{0}} \ar@/_1.5pc/[l]_{\l_0}
\ar@/^1.5pc/[l]^{\r_0} & B},\quad \partial_{0}\partial_{1}=0,
\end{equation}
 which is now easily compared with the one of an internal 2-category (taking $\lambda_0=\rho_0=0$). In the literature some other generalizations of the notion of internal 2-category are considered whose main motivations come from homotopy theory.  For instance in \cite{Akca} the authors introduce the notion of pseudo 2-crossed module relating it with pseudo simplicial groups.

Having a basic structure such as (\ref{diag-1}) we now show how to recover, up to isomorphism, the original structure of internal bicategory (assuming that we use multiplicative notation for  $B$ and $Z$  while $Y$ is written additively, although it is not necessarily commutative):
\begin{enumerate}
\item the objects of the bicategory are the elements in $B$;
\item the  morphisms of the bicategory are the pairs $(z,b)\in Z\times B$ with domain and codomain as illustrated\[\xymatrix{\partial_0(z)b&b\ar[l]_(.35){(z,b)}};\]
\item the 2-cells of the bicategory are the triples $(y,z,b)\in Y\times Z\times B$, which can be
displayed as
\[\xymatrix{
\ar@{}[r]|(1){\partial_{0}(z)b}& &&
\ar@<1ex>@/^1pc/[ll]^{(\partial_{1}(y)z,b)}="2"
\ar@<-1ex>@/_1pc/[ll]_{(z,b) }="4"     \ar@{}[r]|(0){b} & \\
\ar@{=>}|{(y,z,b) } "4";"2"}\] thus illustrating their source and target associated morphisms;
\item the vertical composition is simply obtained by the formula \[(y',\partial_1(y)z,b)\circ(y,z,b)=(y'+y,z,b),\]
as illustrated,
\[\xymatrix{
\ar@{}[r]|(.6){\partial_{0}(z)b}& &&
\ar@<0ex>@/^0pc/[ll]|{(\partial_{1}(y)z,b)}="2"
\ar@<1.5ex>@/^2pc/[ll]^{(\partial_{1}(y')z',b)}="6"
\ar@<-1.5ex>@/_2pc/[ll]_{(z,b) }="4"     \ar@{}[r]|(0){b} & \\
\ar@{=>}|{(y,z,b) } "4";"2"
\ar@{=>}|{(y',z',b) } "2";"6"}\] which is well defined because $z'=\partial_1(y)z$ and $\partial_{1}(y'+y)z=\partial_{1}(y')z'$;
\item the tensor (or horizontal) composition
\[\xymatrix{
\ar@{}[r]|(.9){\partial_{0}(z')b'}& &&
\ar@<1ex>@/^1pc/[ll]^{(\partial_{1}(y')z',b')}="2"
\ar@<-1ex>@/_1pc/[ll]_{(z',b') }="4"     \ar@{}[r]|(0.1){b'=} & \\
\ar@{=>}|{(y',z',b') } "4";"2"}\hspace{-1.5cm}\xymatrix{
\ar@{}[r]|(1){\partial_{0}(z)b}& &&
\ar@<1ex>@/^1pc/[ll]^{(\partial_{1}(y)z,b)}="2"
\ar@<-1ex>@/_1pc/[ll]_{(z,b) }="4"     \ar@{}[r]|(0){b} & \\
\ar@{=>}|{(y,z,b) } "4";"2"}\]
is defined by the
formula
\begin{eqnarray*}
    (y',z',\partial_{0}(z)b)\otimes (y,z,b)&=&(u(y',z')+v(y,z),b) 
\end{eqnarray*}
where $u(y',z')+v(y,z)$ is the element in $X\simeq Y\rtimes_\phi Z$, obtained from $(\ref{eq: u})$ and $(\ref{eq: v})$, which means that it is of the form \[(-\r_0\partial_1(y')+y'+\phi(z'-\partial_1\r_0(z'),-\l_0\partial_1(y)+y), z'-\partial_1\r_0(z')+z-\partial_1\l_0(z));\]
in particular it follows that $(z',b')\otimes(z,b)$, with $b'=\partial_0(z)b$, is the pair \[(z'-\partial_1\r_0(z')+z-\partial_1\l_0(z),b)\in Z\times B;\]
\item Finally, the coherence 2-cells are obtained as follows:
\begin{enumerate}
\item the left and right coherence isomorphism 2-cells, respectively $l(z,b)$ and $r(z,b)$, for the composition with  pseudo-identity morphisms, are given by:
\[\xymatrix{
\ar@{}[r]|(0.6){\partial_0(z)b}& & &&& \ar[llll]|{(z,b)}="3"
\ar@<1ex>@/^2pc/[llll]^{(0,\partial_{0}(z)b)\otimes (z,b)}="2"
\ar@<-1ex>@/_2pc/[llll]_{(z,b)\otimes (0,b) }="4" \ar@{}[r]|(0){b} & \\
\ar@{=>}^{l(z,b) } "2";"3" \ar@{=>}^{r(z,b) } "4";"3"}\]
\begin{align*}
l(z,b)=&(\l_0(z),z-\partial_{1}\l_0(z),b),\\
r(z,b)=&(\r_0(z),z-\partial_{1}\r_0(z),b);
\end{align*}
\item the associativity coherence 2-cell isomorphism $a(z'',z',z,b)$, as displayed,\[\xymatrix{
\ar@{}[r]|(0.6){\partial_0(z'')b''}& & &&&
\ar@<-1ex>@/_1pc/[llll]_{(z'',b'')\otimes((z',b')\otimes(z,b))}="2"
\ar@<1ex>@/^1pc/[llll]^{((z'',b'')\otimes(z',b'))\otimes(z,b) }="4" \ar@{}[r]|(0){b} & \\
\ar@{=>}|{a(z'',z',z,b) } "2";"4"}\]
is given by the formula
\begin{align*}
a(z'',z',z,b)=&(\a_1(0,z'')+\a_2(0,z')+\a_3(0,z), \bar{z} ,b)
\end{align*}
where $\bar{z}=(z'',b'')\otimes((z',b')\otimes(z,b))$ while the $\alpha_i$, $i=1,2,3$, are  obtained from the equations  $(\ref{eq: a1})$ to $(\ref{eq: a3})$ with
\begin{eqnarray}
s(y,z)&=&(0,z)\nonumber\\
t(y,z)&=&(0,\partial_1(y)z)\nonumber\\
\lambda(y,z)&=&(y+\lambda_{0}(z),z-\partial_{1}\lambda_{0}(z))\nonumber\\
\rho(y,z)&=&(y+\rho_{0}(z),z-\partial_{1}\rho_{0}(z))\nonumber
\end{eqnarray}
as already remarked.
\end{enumerate}
\end{enumerate}

In the case when $\l_0=\r_0=0$, a structure equivalent to the one of a
2-crossed module of groups  is obtained \cite{Conduche}.

\section{Example}

If $B$ is a topological abelian group then we can always construct an internal bicategory structure over it which  coincides, up to homotopy, with the bicategory of paths in $B$. Using the same notation as in the previous section, we define $Z$ as the set of continuous maps
\[z\colon{[0,1] \to B}\] such that $z(0)=0$, with the componentwise addition \[(z+z')(t)=z(t)+z'(t),\] and, similarly, we let $Y$ to be the set of continuous maps
\[y\colon{[0,1]\times [0,1] \to B}\] such that $y(0,t)=0$ and
$y(s,1)=y(s,0)=0$ for every $s,t\in [0,1]$.

The homomorphisms displayed in diagram (\ref{diag-1}) are the following ones:
\begin{eqnarray*}
\partial_0(z)&=&z(1)\\
\partial_1(y)(t)&=&y(1,t)\\
\rho_0(z)(s,t)&=&\left\lbrace\begin{matrix}
z(t)& \text{if} & 2t\leq s\\
z(t)-z\left(\frac{2t-s}{2-s}\right)&\text{if}&2t>s
\end{matrix}\right.\\
\lambda_0(z)(s,t)&=&\left\lbrace\begin{matrix}
z\left(\frac{st}{2-s}\right)-z\left(\frac{2t}{2-s}\right)&\text{if}&2t\leq s\\
z\left(\frac{st}{2-s}\right)-z\left(\frac{s}{2-s}\right)&\text{if}&2t>s
\end{matrix}\right.\\
\end{eqnarray*}

Finally, we observe that all the actions involved are trivial and the formula for horizontal composition of paths, obtained from item 5 of the previous section, \[(z',b')\otimes (z,b)=(z'-\partial_1\rho_0(z')+z-\partial_1\lambda_0(z),b)\] with $b'=\partial_0(z)+b$, gives precisely the usual formula for the concatenation of paths \[(z'(t)+b')\otimes(z(t)+b)=\left\lbrace\begin{matrix}
z(2t)+b&\text{if}&2t\leq 1\\
z'(2t-1)+z(1)+b&\text{if}&2t>1
\end{matrix}\right.\]in which the pair $(z,b)$ is identified with the path $z(t)+b$.

We point out that the associativity isomorphism $\a$, defined at item~6 of the previous section, is completely determined by the pseudo-identity isomorphisms $\lambda$ and $\rho$. This is a major characteristic for every structure of an internal bicategory in groups and, more generally, in any (weakly) Mal'tsev category \cite{NMF-PhD}. In the following sections we work towards a characterization of internal bicategory in a weakly Mal'tsev category with a $V$-Mal'tsev operation. The result is an improvement of Theorem 6.2 in \cite{NMF-2cells}. Indeed, the $V$-Mal'tsev operation allows us to write  the admissibility condition (in the sense of \cite{NMF-PhD}) via an explicit formula. This is not possible for arbitrary weakly Mal'tsev categories and it is similar to the case of Mal'tsev varieties of universal algebra \cite{Janelidzepreprint}.

\section{Internal groupoids in weakly Mal'tsev categories with a $V$-Mal'tsev operation}\label{sec:5}

In this section we use the notion of Mal'tsev operator introduced in \cite{Pedicchio} and obtain a large class of examples of weakly Mal'tsev categories which are not necessarily Mal'tsev (see also \cite{MF:14online}).

Let $\A$ and $\B$ be two categories with finite limits and $V\colon{\A\to \B}$ a functor that preserves finite limits.

\begin{definition}[Pedicchio,\cite{Pedicchio}]
A $V$-Mal'tsev operation on $\A$ is a  natural transformation $p:V^3\to V$ whose components are morphisms in $\B$ of the form \[p_A\colon{V(A)\times V(A)\times V(A)\to V(A)}\] such that \begin{equation}\label{cond:mal'tsev}p_A\langle x,y,y\rangle=x=p_A\langle y,y,x\rangle\end{equation} for all $B$ in $\B$ and $x,y\colon{B\to V(A)}$.
\end{definition}

\begin{proposition}[Pedicchio,\cite{Pedicchio}] If $V$ is conservative  and $\A$ admits  a $V$-Mal'tsev operation then  $\A$ is a Mal'tsev category.
\end{proposition}

Recall that a functor is said to be conservative if it is faithful and reflects isomorphisms. The author is grateful to George Janelidze for the inspiration that led to the following result.

\begin{proposition} If $V$ is faithful and $\A$ admits a $V$-Mal'tsev operation then $\A$ is a weakly Mal'tsev category.
\end{proposition}
\begin{proof}
The proof uses essentially the same arguments as in \cite{MF:14online}, Section 2.
We have to prove that for every diagram in $\A$ of the form
\[\xymatrix{ A \ar@<.5ex>[r]^{f} & B \ar@<.5ex>[l]^{r}
\ar@<-.5ex>[r]_{s} & C \ar@<-.5ex>[l]_{g} }\] with $fr=1_B=gs$, the
induced morphisms $e_1=\langle 1_A,sf\rangle$ and $e_2=\langle
rg,1_C\rangle$ into the pullback
\begin{equation}\label{pullback01}
\xymatrix{{A\times_B C}\ar[r]^(.65){\pi_2}\ar[d]_{\pi_1}&{C}\ar[d]^{g}\\{A}\ar[r]^{f}&{B}}
\end{equation}
are jointly
epimorphic. In other words, given two morphisms in $\A$, \[\varphi,\varphi'\colon{A\times_B A \to D}\] such that $\varphi e_1=\varphi'e_1$ and $\varphi
e_2=\varphi' e_2$, we have to prove that $\varphi=\varphi'$.  Since $V$ is faithful it is sufficient to
 prove $V(\varphi)=V(\varphi')$.
 Indeed, in terms of generalized elements,
 \begin{eqnarray*}
V\varphi(a,c)&=&p_D(V\varphi(a,c),V\varphi(rf(a),sg(c)),V\varphi(rf(a),sg(c)))\\
&=& V\varphi(p_A(a,rf(a),rf(a)),p_C(c,sg(c),sg(c)))\\
&=& V\varphi(p_A(a,rf(a),rf(a)),p_C(sg(c),sg(c),c))\\
&=& p_D(V\varphi(a,sg(c)),V\varphi(rf(a),sg(c)),V\varphi(rf(a),c))\\
&=& p_D(V\varphi(a,sf(a)),V\varphi(rf(a),sf(a)),V\varphi(rg(c),c))\\
&=& p_D(V\varphi e_1(a), V\varphi e_1 rf(a), V\varphi e_2 (c)) \\
&=& p_D(V\varphi' e_1(a), V\varphi' e_1 r f(a), V\varphi' e_2 (c))
 \end{eqnarray*}
 and hence $V\varphi=V\varphi'$, so that $(e_1,e_2)$ is a jointly
 epimorphic pair of morphisms.

\end{proof}

Note that we sometimes omit unnecessary parenthesis, so for example $V\varphi(a,c)$ means $V(\varphi(a,c))$.

From now on we assume that $V$ is a faithful functor which preserves finite limits and $\A$ admits a Mal'tsev operation $p\colon{V^3\to V}$.

A useful corollary that will be needed later on is:

\begin{corollary} For every diagram in $\A$ of the form
\[\xymatrix{ A \ar@<.5ex>[r]^{f} \ar[rd]_{\a} & B \ar[d]^{\b} \ar@<.5ex>[l]^{r}
\ar@<-.5ex>[r]_{s} & C \ar[ld]^{\g} \ar@<-.5ex>[l]_{g} \\ & D }\]
with $fr=1_B=gs$ and $\a r=\b=\g s$, there exists at most one
morphism $\varphi\colon{A\times_B C\to D}$ such that \[\varphi
e_1=\a \hbox{ and } \varphi e_2=\g.\] Furthermore, when such a
morphism exists, it is such that \[V(\varphi)=p_D\langle
V(\a\pi_1),V(\b f\pi_1),V(\g\pi_2)\rangle,
\] where $\pi_1$ and $\pi_2$ are the canonical projections in the pullback diagram (\ref{pullback01}).
\end{corollary}

In order to describe internal pseudocategories in $\A$ we first need to say what are internal groupoids,  internal functors and  internal natural transformations. To simplify notation, instead of $\Grpd(\A)$ we consider a different (but equivalent)  category that will be denoted by $\G(\A)^{*}$.

The category $\G(\A)$ has objects all triples $(A,s,t)$ with $A$ an object in $\A$ and $s,t\colon{A\to A}$ two endomorphisms of $A$ such that $st=t\text{ and } ts=s$. Observe that, as a consequence, we  have that $ss=s$ and $tt=t$. Indeed $ss=s(ts)=(st)s=ts=s$ and similarly for $tt=t$. The morphisms, say $f\colon{(A',s',t')\to(A,s,t)}$, are the morphisms $f\colon{A'\to A}$ in $\A$ such that $sf=fs'$ and $tf=ft'$.

The category $\G(\A)^{*}$ is defined as the full subcategory of $\G(\A)$ consisting of those triples $(A,s,t)$ in $\G(\A)$ for which the following property (P1) holds.
\begin{enumerate}
\item[(P1)]  for
every three morphisms in $\A$, \[\a,\b,\g\colon{X\to A}\] with
$s\a=s\b$ and $t\b=t\g$, there exists a unique morphism in $\A$,
$\varphi\colon{X\to A}$, such that \[V(\varphi)=p_A\langle
V\a,V\b,V\g\rangle.\]
\end{enumerate}

In order to prove that $\Grpd(\A)\cong\G(\A)^{*}$ we observe an important consequence of the property (P1) for a given object $(A,s,t)$ in $\G(\A)$.
See \cite{MF+VdL:14a} for further details on this particular instance for the associative property of the $V$-Mal'tsev operation $p$.

\begin{proposition}
Suppose $(A,s,t)$ is an object in $\G(\A)$ for which the property (P1) holds. Then the $V$-Mal'tsev operation $p$ is associative in the sense that for every nine morphisms $$x_i,y_i,z_i\colon{X\to A},\quad i=1,2,3,$$ in $\A$, such that
\begin{eqnarray}
sx_1=sx_2&=&sy_2=sy_1\label{cond10}\\
tx_2&=&tx_3\\
ty_1&=&tz_1\\
sx_3&=&sy_3\\
sz_1&=&sz_2\\
ty_3=tz_3&=&ty_2=tz2,\label{cond15}
\end{eqnarray}
the following equation holds
\begin{gather}\label{passoc}
p_A\langle p_A\langle
Vx_1, Vx_2, Vx_3\rangle, p_A\langle
Vy_1, Vy_2, Vy_3 \rangle,p_A\langle
Vz_1, Vz_2, Vz_3 \rangle\rangle
=\\
p_A\langle p_A\langle
Vx_1, Vy_1, Vz_1 \rangle,p_A\langle
Vx_2, Vy_2, Vz_2 \rangle,p_A\langle
Vx_3, Vy_3, Vz_3 \rangle\rangle.\nonumber
\end{gather}
\end{proposition}
\begin{proof}
Let us first observe that the unique morphism $$\varphi=\varphi(x_1,x_2,x_3)\colon{X\to A}$$ for which $V(\varphi)=p_A\langle Vx_1,Vx_2,Vx_3\rangle$ is such that $s\varphi=sx_3$ and $t\varphi=tx_1$. The same is true for all the other combinations of appropriate $x_i$, $y_i$, $z_i$. This means that if we consider the iterated pullback diagram (compare with the notion of pre-groupoid in \cite{MF+VdL:14a})
\[\xymatrix{R(s,t)\ar[r]^{\pi_2}\ar[d]_{\pi_1}&R(t)\ar[r]^{t_2}\ar[d]_{t_1}& A\ar[d]^{t}\\R(s)\ar[r]^{s_2}\ar[d]_{s_1}&A\ar[r]^{t}\ar[d]^{s}&A\\A\ar[r]^{s}&A}\]
in which all squares are pullback squares we have that each one of the following ordered triples $(x_1,x_2,x_3)$, $(y_1,y_2,y_3)$, $(z_1,z_2,z_3)$, $(x_1,y_1,z_1)$, $(x_2,y_2,z_2)$, $(x_3,y_3,z_3)$, $(\varphi(x_1,x_2,x_3), \varphi(y_1,y_2,y_3), \varphi(z_1,z_2,z_3))$ and $$(\varphi(x_1,y_1,z_1), \varphi(x_2,y_2,z_2), \varphi(x_3,y_3,z_3))$$ induces a morphism from $X$ into $R(s,t)$.

We now observe that the naturality  of $p$ with respect to the morphism $s_1\pi_1$ gives the commutative square
\[\xymatrix{VR(s,t)^3\ar[r]^{p_{R(s,t)}}\ar[d]_{V(s_1\pi_1)^3}& VR(s,t)\ar[d]^{V(s_1\pi_1)}\\VA^3\ar[r]^{p_A}&VA}\] and that  similar commutative squares are obtained for the morphism $s_2\pi_1$ (which is equal to the morphism $t_1\pi_2$) and for the morphism $t_2\pi_2$. This shows that  the  morphism $p_{R(s,t)}$ is equal to the induced morphism \[\langle p_AV(s_1\pi_1)^3,p_AV(s_2\pi_1)^3, p_AV(t_2\pi_2)^3\rangle,\] indeed $VR(s,t)\cong R(V(s),V(t))$. Now, the property (P1) gives us a unique morphism $\pi\colon{R(s,t)\to A}$ such that \[V(\pi)=p_A\langle Vs_1\pi_1,Vs_2\pi_1,Vt_2\pi_2\rangle,\]
which, by naturality of $p$, gives the  commutativity of the square
\[\xymatrix{VR(s,t)^3\ar[r]^{p_{R(s,t)}}\ar[d]_{V(\pi)^3}& VR(s,t)\ar[d]^{V(\pi)}\\VA^3\ar[r]^{p_A}&VA.}\]
From here we conclude that the morphism\[p_AV(\pi)^3\langle
\langle Vx_1, Vx_2, Vx_3\rangle, \langle
Vy_1, Vy_2, Vy_3, \rangle,\langle
Vz_1, Vz_2, Vz_3 \rangle\rangle\]
is equal to
\[V(\pi)p_{R(s,t)}
\langle
\langle Vx_1, Vx_2, Vx_3\rangle, \langle
Vy_1, Vy_2, Vy_3, \rangle,\langle
Vz_1, Vz_2, Vz_3 \rangle\rangle
\]
which simplifies to the equation (\ref{passoc}) and completes the proof.
\end{proof}

Two useful particular cases are
\begin{gather}\label{passocleft}
p_A\langle
Vx_1, Vx_2, p_A\langle
Vx_2, Vy_2, Vz_3 \rangle\rangle
=
 p_A\langle
Vx_1, Vy_2, Vz_3 \rangle
\end{gather}
and
\begin{gather}\label{passocright}
p_A\langle
p_A\langle
Vx_1, Vy_2, Vz_2, \rangle,
 Vz_2, Vz_3 \rangle
=
p_A\langle
Vx_1, Vy_2, Vz_3 \rangle\rangle.
\end{gather}

The following result makes explicit the categorical equivalence between $\Grpd(\A)$ and $\G(\A)^{*}$.

\begin{proposition}
There is an equivalence of categories between $\Grpd(\A)$ and $\G(A)^{*}$.
\end{proposition}
\begin{proof}
In order to prove that $\Grpd(\A)\cong\G(\A)^{*}$ we construct two functors \[F\colon{\Grpd(\A)\to\G(\A)^{*}}\] and \[G\colon{\G(\A)^{*}\to \Grpd(\A)}\]
and show that, up to isomorphism, they are inverse to each other. The functor $F$ takes an internal groupoid
\[C=\left(\xymatrix{C_2
\ar@<1.5ex>[r]^{\pi_2}\ar@<-1.5ex>[r]_{\pi_1}\ar@<0ex>[r]|{m} & C_1 \ar@(ru,lu)_{i}
\ar@<1ex>[r]^{d}\ar@<-1ex>[r]_{c}
& C_0 \ar@<0ex>[l]|{e} }\right)\] in which $C_2$ is defined together with $\pi_1$ and $\pi_2$ as the pullback of $d$ along $c$, $m$ is the morphism giving the composition of arrows and $i$ gives the inverses (see \cite{Janelidzepreprint} for more details), and sends it to the triple $(C_1,ed,ec)$ in $G(\A)^{*}$. The conditions $st=t$ and $ts=s$ are easily obtained from the reflexivity condition $de=1_{C_1}=ce$. Let us check that the property (P1) holds. Given any three morphisms $\alpha,\beta,\gamma\colon{X\to C_1}$ such that  $ed\alpha=ed\beta$ and $ec\beta=ec\gamma$ or, equivalently, \[d\alpha=d\beta\quad\text{ and }\quad c\beta=c\gamma,\] we have an induced morphism \[\langle\alpha,\beta,\gamma \rangle\colon{X\to R(d,c)}\]into the iterated pullback R(d,c) obtained as follows
\[\xymatrix{R(d,c)\ar[r]^{\pi_2}\ar[d]_{\pi_1}&R(c)\ar[r]^{c_2}\ar[d]_{c_1}& C_1\ar[d]^{c}\\R(d)\ar[r]^{d_2}\ar[d]_{d_1}&C_1\ar[r]^{c}\ar[d]^{d}&C_1\\C_1\ar[r]^{d}&C_0.}\]
It is now clear that the desired morphism $\varphi\colon{X\to C_1}$ for which $$V(\varphi)=p_{C_1}\langle V\alpha,V\beta,V\gamma\rangle$$ is
\[\varphi=m\langle \alpha,m\langle i\beta,\gamma\rangle\rangle.\]
Indeed, from \cite{Janelidzepreprint} we know that $V(i)=p_{C_1}\langle V(ed),1_{VC_1},V(ec)\rangle$ and \[V(m)=p_{C_1}\langle V(\pi_1),V(ed\pi_1),V(\pi_2)\rangle\]
so that we observe
\begin{gather*}
V(m\langle \alpha,m\langle i\beta,\gamma\rangle\rangle)=\\
p_{C_1}\langle V(\alpha),V(ed\alpha), p_{C_1}\langle p_{C_1}\langle V(ed\beta),V(\beta),V(ec\beta)\langle,V(edi\beta),V(\gamma)\rangle\rangle
\end{gather*}
and because $di=c$ and $ed\alpha=ed\beta$ we also have
\begin{gather*}
V(m\langle \alpha,m\langle i\beta,\gamma\rangle\rangle)=\\
p_{C_1}\langle V(\alpha),V(ed\beta), p_{C_1}\langle p_{C_1}\langle V(ed\beta),V(\beta),V(ec\beta)\langle,V(ec\beta),V(\gamma)\rangle\rangle.
\end{gather*}
At last, using (\ref{passocleft}) and (\ref{passocright})  we conclude that
\[V(m\langle \alpha,m\langle i\beta,\gamma\rangle\rangle)=p_{C_1}\langle V\alpha,V\beta,V\gamma\rangle.\]

Conversely, given $(A,s,t)$ in $\G(\A)^{*}$ we obtain an internal groupoid
\[G(A,s,t)=\left(\xymatrix{C_2
\ar@<1.5ex>[r]^{\pi_2}\ar@<-1.5ex>[r]_{\pi_1}\ar@<0ex>[r]|{m} & C_1 \ar@(ru,lu)_{i}
\ar@<1ex>[r]^{d}\ar@<-1ex>[r]_{c}
& C_0 \ar@<0ex>[l]|{e} }\right)\]
as follows:
\begin{enumerate}
\item[(a)] $C_1=A$
\item[(b)] $e\colon{C_0\to C_1}$ is the equalizer of the pair of morphisms $s,1_{C_1}\colon{C_1\to C_1}$
\item[(c)] $d$ and $c$ are the unique morphisms from $C_1$ into $C_0$ such that
\begin{eqnarray}
ed=s\nonumber\\
ec=t\nonumber
\end{eqnarray}
which exist because $st=t$ and $ss=s(ts)=(st)s=ts=s$, as illustrated
\[\xymatrix{ C_1\ar@<.5ex>[rd]^{t}\ar@<-.5ex>[rd]_{s}\ar@<.5ex>@{-->}[d]^{c}\ar@<-.5ex>@{-->}[d]_{d}\\C_0\ar[r]_{e}&C_1\ar@<.5ex>[r]^(.35){s}\ar@<-.5ex>[r]_(.35){1_{C_1}}&C_1=A}\]
\item[(d)] $C_2$ is obtained together with $\pi_1$ and $\pi_2$ as the pullback of $d$ along $c$ as pictured \begin{equation}\label{pullback02}
\xymatrix{C_2\ar[r]^(.5){\pi_2}\ar[d]_{\pi_1}&{C_1}\ar[d]^{c}\\{C_1}\ar[r]^{d}&{C_0}}
\end{equation}
\item[(e)] the existence of the morphisms $i\colon{C_1\to C_1}$ and $m\colon{C_2\to C_1}$ is a consequence of (P1) since (as it is explained in \cite{Janelidzepreprint} for Mal'tsev varieties) \[V(i)=p_{C_1}\langle V(ed), V(1_{C_1}),V(ec)\rangle\] and \[V(m)=p_{C_1}\langle V(\pi_1), V(ed\pi_1),V(\pi_2)\rangle.\]
\end{enumerate}

On morphisms we have $F(f_2,f_1,f_0)=f_1$ and $G(f)=(f\times_{dfe} f, f, dfe)$ and it is easy to check that all the required conditions are satisfied on both sides. Moreover it is not difficult to observe that $GF\cong 1_{\Grpd(\A)}$ and that $FG\cong 1_{\G(\A)^{*}}$.
\end{proof}

\section{Internal pseudocategories}

In this section we use the results from \cite{NMF-2cells} and give an explicit description for the notion of pseudocategory in $\G(\A)^{*}$. We continue to assume that $\A$ is a weakly Mal'tsev category with a $V$-Mal'tsev operation $p\colon{V^3\to V}$ and that $V\colon{\A\to\B}$ is a faithful functor which preserves finite limits.

To be able to use the results from \cite{NMF-2cells} we have to equip the category $\G(\A)^{*}$ with a two cell structure (\cite{NMF-2cells}, Definition 3.1). It is not difficult to observe that the equivalence \[\Grpd(\A)\cong\G(\A)^{*}\]
induces a natural 2-cell structure on $\G(\A)^{*}$. We simply have to translate the notion of internal natural transformation from $\Grpd(\A)$ into $\G(\A)^{*}$ by means of the equivalence. Indeed, closely following the  Example 5.7 in \cite{NMF-2cells}, if $\tau_0\colon{C_0\to C'_1}$ is an internal natural transformation in $\Grpd(\A)$, as illustrated,
\begin{equation}\label{diag:tau0}\xymatrix{C_2
\ar@<1.5ex>[r]^{\pi_2}\ar@<-1.5ex>[r]_{\pi_1}\ar@<0ex>[r]|{m} & C_1 \ar@(ru,lu)_{i}
\ar@<1ex>[r]^{d}\ar@<-1ex>[r]_{c}
& C_0 \ar@<0ex>[l]|{e} \ar[ld]_(.6){\tau_0}
\\
C'_2
\ar@<1.5ex>[r]^{\pi'_2}\ar@<-1.5ex>[r]_{\pi'_1}\ar@<0ex>[r]|{m'} & C'_1 \ar@(rd,ld)^{i'}
\ar@<1ex>[r]^{d'}\ar@<-1ex>[r]_{c'}
& C'_0 \ar@<0ex>[l]|{e'}
}\end{equation}
from an internal functor $(f_1\times_{f_0}f_1, f_1,f_0)$ to  $(g_1\times_{g_0} g_1,g_1,g_0)$, then the triple $(g_1,\tau_0d,f_1)$ is a 2-cell in $\G(\A)^{*}$. In more general terms, if $$g,f\colon{(A',s',t')\to(A,s,t)}$$ are two parallel morphisms in $\G(\A)^{*}$ then a 2-cell from $f$ to $g$ is completely determined by a morphism $\tau\colon{A'\to A}$ with $s\tau=sf$, $t\tau=tg$ and such that
\[V(\tau)=p_A\langle Vg,Vt\tau s,V\tau
s\rangle=\langle V\tau t, Vs\tau t,Vf\rangle.\]
In fact, the morphism $\tau$ alone determines the morphisms $f$ and $g$.

\begin{proposition}
Given any three parallel morphisms $$\tau,f,g\colon{(A',s',t')\to(A,s,t)}$$ in $\G(\A)^{*}$ such that $s\tau=sf$ and $t\tau=tg$, if
\[V(\tau)=p_A\langle V(g),V(t\tau s),V(\tau
s)\rangle=p_A\langle V(\tau t), V(s\tau t),V(f)\rangle\]
then $f$ and $g$ are uniquely determined by
\[V(g)=p_A\langle V(\tau ),V(\tau s),V(t\tau s)\rangle\]
\[V(f)=p_A\langle V(s\tau t),V(\tau t),V(\tau)\rangle.\]
\end{proposition}
\begin{proof}
We observe that
\[p_A\langle V(\tau),V(\tau s),V(t\tau s)\rangle=p_A\langle p_A\langle V(g),V(t\tau s),V(\tau
s)\rangle, V(\tau s), V(t\tau s)\rangle\] which by (\ref{passocright}) becomes
\[p_A\langle V(\tau),V(\tau s),V(t\tau s)\rangle= p_A\langle V(g),V(t\tau s),V(t\tau s)\rangle\] and by (\ref{cond:mal'tsev}) is simply
\[p_A\langle V(\tau),V(\tau s),V(t\tau s)\rangle=  V(g).\]
In a similar way we prove the other condition.
\end{proof}

The previous result tells us that, for a fixed pair of objects $$((A',s',t'),(A,s,t))$$ in ${\G(\A)^{*}}^{op}\times \G(\A)^{*}$, every morphism $\tau\colon{A'\to A}$ in $\A$, for which $t\tau t'=t\tau$ and $s\tau s'=s\tau$, can be regarded as a 2-cell. Its domain is the unique morphism, say $\dom(\tau)$, in $\A$ for which \[V(\dom(\tau))=p_A\langle V(s\tau t),V(\tau t),V(\tau)\rangle\] while its codomain is the unique morphism in $\A$, say $\cod(\tau)$, for which
\[V(\cod(\tau))=p_A\langle V(\tau ),V(\tau s),V(t\tau s)\rangle.\]
The existence of the two morphisms is ensured  by the property (P1) and it is not difficult to check that $\dom(\tau)$ and $\cod(\tau)$ are morphisms in $\G(\A)^{*}$.
This gives us a bifunctor \[H\colon{{\G(\A)^{*}}^{op}\times \G(\A)^{*}\to\Set}\]
together with two natural transformations \[\dom,\cod\colon{H\to \hom_{\G(\A)^{*}}}\]
and moreover we can define two other natural transformations, $0$ and $+$, and make the structure $(H,\dom,\cod,0,+)$ a 2-cell structure on the category $\G(\A)^{*}$ as defined in \cite{NMF-2cells}, Definition 3.1. The natural transformation $0\colon{\hom\to H}$ is such that $0(f)=f$ and, given any two elements $\tau,\sigma$ in $H((A',s',t'),(A,s,t))$ such that $\dom(\tau)=f=\cod(\sigma)$, then $\tau+\sigma$ is the unique morphism in $\A$, $\tau+\sigma\colon{A'\to A}$, for which \[V(\tau+\sigma)=p_A\langle V(\tau),V(f),V(\sigma)).\]

The following result summarizes the previous statement.

\begin{proposition} Consider the structure $(H,\dom,\cod,0,+)$, with $H$ the bifunctor \[\xymatrix{(A''',s''',t''')&(A'',s'',t'')\ar[d]^{v}&\ar@{}[d]="a"&\ar@{}[d]="b"&H((A''',s''',t'''),(A'',s'',t''))\ar[d]^{H(u,v)}\\(A',s',t')\ar[u]^{u}&(A,s,t)&&&H((A',s',t'),(A,s,t))\ar@<4ex>@{|->}|{} "a";"b"}\] from ${\G(\A)^{*}}^{op}\times \G(\A)^{*}$ into the category of Sets given by\[H((A',s',t'),(A,s,t))=\{\tau\colon{A'\to A}\mid s\tau s'=s\tau,t\tau t'=t\tau\}\]
\[H(u,v)(x)=vxu,\]
and the natural transformations $\dom$, $\cod$, $0$ and $+$ defined as  above. This structure is an invertible,
natural and cartesian 2-cell structure on $\G(\A)^{*}$.
\end{proposition}
\begin{proof}
First we show that $H(u,v)(x)$ is well defined. Indeed $s(vxu)s'=vs''xs'''u$ and, since $s''xs'''=s''x$, we have that $s(vxu)s'=vs''xu=svxu$. Similarly, we prove that $t(vxu)t'=t(vxu)$. Now we prove that $\dom$ and $\cod$ are well defined. Indeed,
\begin{eqnarray*}
V(s\dom(\tau))&=&V(s)V(\dom(\tau))=V(s)p_A\langle V(s\tau t'),V(\tau t'),V(\tau)\rangle\\
&=&p_A\langle V(ss\tau t'),V(s\tau t'),V(s\tau)\rangle\\
&=&p_A\langle V(s\tau t'),V(s\tau t'),V(s\tau)=V(s\tau)
\end{eqnarray*}
 means that $s\dom(\tau)=s\tau$. Since $\tau$ is in $H((A',s',t'),(A,s,t))$ we have $s\dom(\tau)=s\tau s'$, and to prove that $s\dom(\tau)=\dom(\tau)s'$ we observe
\begin{eqnarray*}
V(\dom(\tau)s')&=&p_A\langle V(s\tau t's'),V(\tau t's'),V(\tau s')\rangle\\
&=&p_A\langle V(s\tau s'),V(\tau s'),V(\tau s')\rangle=V(s\tau s').\\
\end{eqnarray*}
A similar argument shows that $t\dom(\tau)=\dom(\tau)t'$ and similarly we can prove that $\cod$ is well defined. Moreover, an analogous argument can be used to prove that $0$ and $+$ are well defined. This shows that the structure is well defined. To see that it is invertible we have to exhibit an inverse $-\tau$ for each element $\tau$ in $H((A',s',t'),(A,s,t))$. In fact $-\tau\colon{A'\to A}$ is the unique morphism in $\A$ for which  $V(-\tau)=p_A\langle V(\dom(\tau)),
V\tau,V(\cod(\tau))\rangle$ whose existence is guaranteed by the property (P1). It remains to prove that the structure is a natural and cartesian 2-cell structure in the sense of Definition 3.1 in \cite{NMF-2cells}. To do that we take $\tau_0$ in the diagram (\ref{diag:tau0}) to be $\tau e'$, with $\tau\in H((A',s',t'),(A,s,t))$ and $e'$ the equalizer of the pair $(1_A,s')$, and observe that it gives an internal natural transformation in $\Grpd(\A)$; more specifically a natural transformation from $G(\dom(\tau))$ into $G(\cod(\tau))$, with $G\colon{\G(\A)^{*}\to\Grpd(\A)}$ the functor  in the equivalence $\G(\A)^{*}\cong\Grpd(\A)$. The result then follows from Example 5.7 in \cite{NMF-2cells}.
\end{proof}

We can now use Theorem 6.2 of \cite{NMF-2cells} in order to give a description of (homogeneous) pseudocategories in $\G(\A)^{*}$. Recall that a homogeneous pseudocategory is a pseudocategory with the extra condition $1_a\otimes 1_a=1_a$ which is reflected as condition $\lambda e=e=\rho e\nonumber$ in the following result.

\begin{proposition}
Let $\A$ be a weakly Mal'tsev category with a $V$-Mal'tsev operation $p$. An internal (homogeneous) pseudocategory in $\Grpd(\A)\cong\G(\A)^{*}$ is completely determined by a diagram in $\G(\A)^{*}$ of the form
\begin{equation*}
\xymatrix{ C_1   \ar@<1ex>[r]^{d} \ar@<-1ex>[r]_{c}  & C_0  \ar[l]|{e}},
\end{equation*}
together with 2-cells $\lambda,\rho\in H(C_1,C_1)$
such that
\begin{gather}
de=1_B=ce\nonumber\\
d\lambda=d=d\rho\nonumber\\
c\lambda=c=c\rho \nonumber\\
\lambda e=e=\rho e\nonumber\\
\cod(\lambda)=1_{C_1}=\cod(\rho)\nonumber
\end{gather}
moreover, there exists a (unique) morphism $m\colon{C_1\times_{C_0}C_1\to C_1}$ in $\G(\A)^{*}$ for which
 \[V(m)=p_{C_1}\langle V(\dom(\rho)\pi_1),V(ed\pi_1),V(\dom(\lambda)\pi_2)\rangle\]
and there exists a (unique) 2-cell $\alpha\in H(C_1\times_{C_0}(C_1\times_{C_0} C_1),
C_1)$ such that \[V(\alpha)=p_{C_1}\langle  V(\a_1\pi_1),V(0_ed\pi_1),p_{C_1}\langle V( \alpha_2\pi_2),V(0_e d\pi_2),V(\alpha_3\pi_3)\rangle \rangle;\]
with
$\a_1,\alpha_2,\alpha_3\in H(C_1, C_1)$ defined as
\begin{eqnarray*}
\alpha_1 &=&\rho \dom(\rho)\\
\alpha_2&=&-\dom(\rho) \lambda +\lambda \dom(\rho)\\
\alpha_3&=&\lambda\dom(\lambda).
\end{eqnarray*}
\end{proposition}
\begin{proof}
The result is a consequence of Theorem 6.2 in \cite{NMF-2cells} or Theorem 108 in \cite{NMF-PhD}, with the admissibility condition translated in terms of the $V$-Mal'tsev operation $p$.
\end{proof}

The previous result can be made more explicit. Indeed we have a concrete description for the objects and the morphisms in the category $\G(A)^{*}$. For example, the object $C_1$ can be seen as a triple $(A,s,t)$ with $A$ and object in $\A$, $s,t\colon{A\to A}$ endomorphisms of $A$ such that $ts=s$, $st=t$ and condition (P1) holds. As a corollary we can take the special case when $C_0$ is of the form $(B,1_B,1_B)$ having then a description of internal bicategories in $\A$.

\begin{proposition} Let $\A$ be a weakly Mal'tsev category with a $V$-Mal'tsev operation $p$. An internal  bicategory (homogeneous with invertible 2-cells)
in $\A$ is completely determined by a diagram in $\A$ of the form
\begin{equation}\label{d.st-lr-dec1}
\xymatrix{ A   \ar@<1ex>[r]^{d} \ar@<-1ex>[r]_{c} \ar@(ld,d)[]_{
\lambda,\rho } \ar@(lu,u)[]^{ s,t }
 & B  \ar[l]|{e}},
\end{equation}
satisfying the conditions
\begin{eqnarray}\label{c.de1ce1}
de=1_B=ce
\end{eqnarray}
\begin{eqnarray}\label{c.st_t_ts_s1}
st=t\quad, \quad ts=s
\end{eqnarray}
\begin{eqnarray}\label{c.ds_d_dt1}
ds=d=dt\nonumber\\
cs=c=ct \\
se=e=te\nonumber
\end{eqnarray}
\begin{eqnarray}\label{c.dro_d_dlambda1}
d\lambda=d=d\rho\nonumber\\
c\lambda=c=c\rho \\
\lambda e=e=\rho e\nonumber
\end{eqnarray}
\begin{eqnarray}\label{c.p1p1}
p_A\langle V\lambda,V\lambda s,V t\lambda s\rangle=1_A=p_A\langle V\rho
,V\rho s,V t\rho s\rangle
\end{eqnarray}
and such that the following properties hold:
\begin{enumerate}
\item [(P1)] for every three morphisms in $\A$, \[\a,\b,\g\colon{X\to
A}\] with $s\a=s\b$ and $t\b=t\g$, there exists a unique morphism in
$\A$, $\varphi\colon{X\to A}$, such that \[V\varphi=p_A\langle
V\a,V\b,V\g\rangle;\]
\item [(P2)] there are two morphisms $m,\mu\colon{A\times_B A\to A}$, in
$\A$, such that \[Vm=p_A(Vu\pi_1,Ved\pi_1,Vv\pi_2)\] and
\[V\mu=p_A(V\a_2\pi_1,Ved\pi_1,V\a_3\pi_2);\] where the morphisms $u,v,\a_2,\a_3\colon{A\to A}$ are
defined by the equations
\begin{eqnarray*}
Vu&=&p_A\langle Vs\r t,V\r t,V\r\rangle \\
Vv&=&p_A \langle Vs\l t, V\l t, V\l \rangle \\
V\a_2 &=& p_A \langle Vs\r\l t, V\r s\l t, p_A\langle Vs\r t, V\r t,
V\r\rangle\rangle \\
V\a_3&=&p_A\langle V\l s\l t, V\l^2 t, V\l^2\rangle
\end{eqnarray*}
\item [(P3)] there exists a morphism $\theta\colon{A\times_B(A\times_B A)\to
A}$, in $\A$, such that \[V\theta=p_A( V\a_1\pi_1,Ved\pi_1,V\mu
\pi_{23});\] the morphism $\mu$ is defined as in (P2) and
$\a_1\colon{A\to A}$ is such that \[V\a_1=p_A\langle Vs\r^2 t, V\r s\r
t, p_A\langle Vs\r t,V \r t,V \r\rangle\rangle.\]
\end{enumerate}

\end{proposition}

As a final remark we observe that a monoidal category is simply a bicategory with one object. Hence, by taking $B$ in the statement of the previous proposition to be the terminal object we obtain a characterization for the internal monoidal categories in $\A$.



\end{document}